\documentclass{amsart}

\usepackage{amssymb, amsmath, latexsym, a4}
\usepackage[latin1]{inputenc}
\usepackage[all]{xy}
\usepackage{enumitem}

\def\1{^{-1}}
\def\id{{\sf id}}

\newtheorem{De}{Definition}

\newtheorem{Pro}[De]{Proposition}
\newtheorem{Le}[De]{Lemma}
\newtheorem{Co}[De]{Corollary}

\newtheorem{Rem}[De]{Remark}

\newtheorem*{Ex*}{Examples}
\newtheorem*{Example*}{Example}
\newtheorem*{th-non}{Theorem}

\def\G{{\bf G}}

\def\cat{{\sf Cat}}

\def\Der{{\sf Der}}

\def\H{{\sf H}}

\def\xto#1{\xrightarrow[]{#1}}
\def\d{\partial}
\def\1{^{-1}}
\def\im{{\sf Im}}
\def\ker{{\sf Ker}}
\def\cok{\sf Coker}

\def\dd{{\delta}}

\begin{document}

	\title{On the centre of crossed modules of  Lie algebras}
	\author{Mariam Pirashvili}
		
	\maketitle
		
\begin{abstract}
	This paper studies the relationship between crossed modules of Lie algebras and their centres. We show that any crossed module \(\partial : L_1\to L_0\) of Lie algebras fits in an exact sequence involving cohomology of the homotopy Lie algebras \(\pi_0(L_*)\) and \(\pi_1(L_*)\).
\end{abstract}

\section{Introduction}

Crossed modules (of groups) were introduced by J. H. C. Whitehead in the 1940's as a tool to study relative homotopy groups $\pi_2(X,A)$ \cite{w2}. 
It was discovered in the 60's that the category of crossed modules is isomorphic to the category of internal categories in the category of groups, see for example \cite{spaces}. Thus crossed modules can be considered as simplifications of such internal categories. A similar simplification exists also for internal categories in the category of Lie algebras. The corresponding objects are known as crossed modules of Lie algebras (see \cite{KL}, \cite{dg1}).

The aim of this work is to introduce the centre of a crossed module of Lie algebras. It is analogous to the centre of a crossed module (of groups) introduced by the author in \cite{mp_cent} which is closely related to the Gottlieb group \cite{gottlieb} of the classifying space and the Drinfeld centre of the corresponding monoidal category \cite{js}.

Recall that a crossed module of Lie algebras can be defined as a linear map $\d:L_1\to L_0$, where $L_0$ is a Lie algebra, $L_1$ is an $L_0$-module, i.e. we are given a bilinear map \(L_0\times L_1\to L_1\), \((x,a)\mapsto x\cdot a\) such that 
\[
[x,y]\cdot a = x\cdot (y\cdot a) - y\cdot (x\cdot a),
\] and $\d$ is a Lie module homomorphism for which the relation 
$$\partial(a)\cdot b+\partial(b)\cdot a=0$$
holds. The essential invariants of a crossed module $L_*=(L_1\xto{\d} L_0)$ are the Lie algebra $\pi_0(L_*)=\cok(\d)$ and the $\pi_0(L_*)$-module $\pi_1(L_*)=\ker(\d)$. Recall also that a braided crossed module of Lie algebras is a linear map of vector spaces $\partial:L_1\to L_0$ together with a bilinear map $L_0\times L_0\to L_1,\ \ (x,y)\mapsto \{x,y\}$
 satisfying the identities (\ref{22}), (\ref{24}) and (\ref{26}), see Section \ref{alternative_def}. Any braided crossed module is also a crossed module, where the Lie algebra structure on $L_0$ and the action of $L_0$ on $L_1$ are given by $$[x,y]=\partial\{x,y\}, \quad x\cdot a=\{x,\partial (a)\}.$$

Now we state our main results.

\begin{th-non}
\begin{enumerate}[wide, labelwidth=!, labelindent=0pt]
	\item [i)]  
Let $\d:L_1\to L_0$ be a crossed module of Lie algebras. There exists a braided crossed module  $\dd: L_1\to {\bf Z}_0(L_*)$, where  ${\bf Z}_0(L_*)$ is the collection of all pairs $(x, \xi)$, where $x\in L_0$ and $\xi:L_0\to L_1$ is a linear map satisfying the following identities 
	\begin{align*}
	\d \xi(t)&=[x,t], \\
	\xi (\d a )&= x \cdot a,  \\
	\xi([s,t])&=  s\cdot \xi(t)-t\cdot \xi(s) .  
	\end{align*}
Here $x,s,t\in L_0$ and $a\in L_1$. The linear map $\dd: L_1\to {\bf Z}_0(L_*)$ is given by
$\dd(c)=(\partial(c),\xi_c)$, where $\xi_c(t)=-t\cdot c.$ Moreover,  the structural bracket ${\bf Z}_0(L_*) \times {\bf Z}_0(L_*) \to L_1$ is given by
$$ \{ (x,\xi),(y,\eta) \}= \xi(y).$$
We call the braided crossed module $\dd: L_1\to {\bf Z}_0(L_*)$ \emph{the centre of the crossed module} $\d:L_1\to L_0$ and  denote it by ${\bf Z}_*(L_*)$.

	\item [ii)] Denote by ${\sf z}_0$ the map ${\bf Z}_0(L_*) \to L_0$ given by ${\sf z}_0(x,\xi)=x$. Define an action of $L_0$ on ${\bf Z}_0(L_*) $ by 
$$y\cdot (x,\xi):= ([y,x],\psi).$$
Here $y\in L_0$, $(x,\xi)\in {\bf Z}_0(L_*) $ and $\psi(t)=t\cdot \xi(y).$ With this action, the map ${\sf z}_0:{\bf Z}_0(L_*) \to L_0$
is a crossed module of Lie algebras which is denoted by $L_*//{\bf Z}_*(L_*)$.

	\item [iii)]  Let $L_*$ be a crossed module of Lie algebras. Then
$$\pi_1( {\bf Z}_*(L_*))\cong \H^0(\pi_0(L_*), \pi_1(L_*))$$
and one has an exact sequence
$$0\to \H^1(\pi_0(L_0), \pi_1(L_*))\to \pi_0({\bf Z}_*(L_*))\to  {\sf Z}_{\pi_1(L_*)}(\pi_0(L_*)) \to \H^2(\pi_0(L_*), \pi_1(L_*)).$$
Here, for a Lie algebra $M$ and an $M$-module $A$, one denotes by ${\sf Z}_A(M)$ the collection of all $m\in M$ such that $[m,x]=0$ for all $x\in M$ (so $m$ is central) and $m\cdot a=0$ for all $a\in A$.
\end{enumerate}
\end{th-non}

Part i) of the Theorem is proved below as Proposition \ref{clb}, while parts ii) and iii) are proved as Proposition \ref{11} and Corollary \ref{12} respectively.

To summarise parts i) and ii), say that any crossed module $\d:L_1\to L_0$ of Lie algebras fits in a commutative diagram

\[\xymatrix{L_1\ar[d]_{id} \ar[r]^{\dd}& {\bf Z}_0(L_*)\ar[d]^{{\sf z}_0}\\
	L_1 \ar[r]^{\d}& L_0}\]
where the top horizontal $L_1\xto{\dd} {\bf Z}_0(L_*)$ and right vertical $ {\bf Z}_0(L_*) \xto {{\sf z}_0} L_0$ arrows have again crossed module structures. In fact, the first one is even a braided crossed module. The pair of maps $({\sf z}_0,\id_{L_1})$ defines a morphism of crossed modules ${\bf Z}_*(L_*)\to L_*$. The crossed module $L_*//{\bf Z}_*(L_*)$ should be understood as the ``homotopic cofibre'' of ${\bf Z}_*(L_*)\to L_*$ since we have the following obvious exact sequence
$$0\to \pi_1({\bf Z}_*(L_*))\to \pi_1(L_*)\to \pi_1(L_*//{\bf Z}_*(L_*))\to $$ 
$$\to \pi_0({\bf Z}_*(L_*))\to \pi_0(L_*)\to \pi_0(L_*//{\bf Z}_*(L_*))\to 0.$$

As we said, the construction of ${\bf Z}_*(L_*)$ and the above properties are parallel to the similar construction given in \cite{mp_cent}. However the centre defined in \cite{mp_cent} has two more important properties. Let us recall these properties, closely following \cite{mp_cent}. Let $\G_*=(\G_1\xto{d} \G_0)$ be a crossed module (of groups). Denote by $\cat(\G_*)$  the corresponding internal category in the category of groups. In particular, it is a monoidal category and hence one can consider the centre ${\bf Z}^{Dr}( \cat(G_*))$ \cite{js}, which is a braided monoidal category and known as the Drinfeld centre. Let us also recall that $B(\G_*)$ denotes the classifying space of $\G_*$. Next, for a topological space $X$ denote by ${\bf Z}(X)$ the connected component of the space $Maps(X,X)$ of self continuous maps $X\to X$ containing the identity map $\id_X$. Then the following assertions hold:
\begin{equation}\label{1} {\bf Z}^{Dr}( \cat(\G_*))\cong \cat({\bf Z}_*(\G_*)), \quad {\bf Z}(B\G_*)\cong B({\bf Z}_*(\G_*)).\end{equation}
Recall also that if $(X,x)$ is a pointed map, then the evaluation at $x$ gives rise to a pointed map $ev_x:({\bf Z}(X), \id_X)\to (X,x)$. Now apply the functor $\pi_1$ and denote by $G(X,x)$ the image of the induced map
$$\pi_1({\bf Z}(X),\id_X)\to \pi_1(X,x).$$
The group $G(X,x)$ is known as the Gottlieb group \cite{gottlieb}.
As a consequence of isomorphisms we proved previously in \cite{mp_cent}, if $X=B(\G_*)$ is the classifying space of a crossed module $\G_*$, then
$G(X,x)$ is a subgroup of the so called Whitehead centre of $X$, which consists of those elements of $\pi_1(X)$ which are central and act trivially on $\pi_2(X)$. 

To find analogues of the isomorphisms (\ref{1}) in the Lie algebra framework would be of great interest. 

\section{Crossed modules in Lie algebras}\label{6sec}

In this section we fix terminology and notation  for  (braided) crossed modules in  Lie algebras \cite{KL}, \cite{CC}. 

\subsection{Definition of crossed and braided crossed modules of Lie algebras}
In this paper we fix a field $k$ of characteristic $\not =2$. All vector spaces and linear maps are considered over $k$. Accordingly, the Lie algebras are defined over $k$. 
 
Recall that if $L$ is a Lie algebra, then a (left) $L$-module is a vector space $V$ together with a bilinear map $L\times V\to V, (x,v)\mapsto x\cdot v$ such that
$$[x,y]\cdot v=x\cdot (y\cdot v)
-y\cdot (x\cdot v).$$
If additionally $V=M$ is also a Lie algebra and 
$$x\cdot [m,n]=[x\cdot m,n]+[m,x\cdot n]$$
holds for all $x\in L$ and $m,n\in M$, then we say that $L$ acts on the Lie algebra $M$. 

\begin{De}  A precrossed module of Lie algebras consists of a homomorphism of Lie
	algebras $\partial: L_1\to L_0$ together with an action of the Lie algebra $L_0$ on $L_1$, denoted by $(x,a)\mapsto x\cdot a$, $x\in L_0$, $a\in L_1$. One requires that the following identity holds:
	 \begin{equation} \label{18}
	 \partial (x\cdot a) = [x, \partial a]\end{equation} for all $a\in L_1$ and $x\in L_0$.
	 If additionally we have the identity \begin{equation}\label{19} \partial(a)\cdot b=[a,b]\end{equation} then $\partial: L_0\to L_0$ is called a crossed module. If additionally the Lie algebra structure on $L_0$ is trivial and the action of $L_0$ on $L_1$ is also trivial, then $L_*$ is called an abelian crossed module.
\end{De}

If $L_*$ is a precrossed module of Lie algebras, then both $Im(\d)$ and $Ker(\d)$ are ideals of $L_0$ and $L_1$  respectively. Thus  $\pi_0(L_*)$ and  $\pi_1(L_*)$ are Lie algebras. If additionally $L_*$ is a crossed module, then $\pi_1$ is a central ideal of $L_1$ and hence $\pi_1(L_*)$ is an abelian Lie algebra and the action of $L_0$ on $L_1$ induces 
a $\pi_0(L_*)$-module structure on $\pi_1(L_*)$. 
\begin{De}
	\cite[Proposition 6.20] {CC}  A braided crossed module (BCM) of Lie algebras $L_*$ consists of a homomorphism of Lie algebras $\partial: L_1\to L_0$ together with a bilinear map $$L_0\times L_0\to L_1, (x,y)\mapsto \{x,y\},$$ such that the
	following  identities hold:
	\begin{align} 
	\partial \{x,y\}&=[x,y], \label{20}\\
	 \{\partial a, \partial b\}&=[a,b],\label{21} \\
	 0&=\{\partial a,x\} +\{x,\partial a\},\label{22}\\
0&=	\{x,[y,z]\}+\{z,[x,y]\}+\{y,[z,x]\}. \label{23}
	\end{align}	
	
\end{De}

\begin{Le} Let  $L_*$ be a BCM of Lie  algebras.
	\begin{enumerate}[wide, labelwidth=!, labelindent=0pt]
		\item [i)]  Define the action of $L_0$ on $L_1$  by
	$$x\cdot a:=\{x,\d(a)\}.$$
	Then $L_*$ is  a crossed module of Lie algebras.
	
	\item [ii)]   The Lie algebra $\pi_0(L_*)$ is abelian and the action of $\pi_0(L_*)$ on $\pi_1(L_1)$ is trivial.
\end{enumerate}
\end{Le}
\begin{proof}
	
\begin{enumerate}[wide, labelwidth=!, labelindent=0pt]
First let us show that we really obtain an action of $L_0$ on the Lie algebra $L_1$. This requires checking the following two identities:
\begin{equation}\label{221} [u,v]\cdot a=u\cdot (v\cdot a)-v\cdot (u\cdot a)\end{equation}
and
\begin{equation}\label{222} x\cdot [a,b]=[x\cdot a,b]+[a,x\cdot b].\end{equation}
In fact, we have
$$	[u,v]\cdot a=\{[u,v],\d(a)\}.$$
On the other hand, we also have
$$u\cdot (v\cdot a)=\{u,\d(v\cdot a)\}=\{u,\d \{v,\d(a)\} \}=\{u, [v,\d(a)]\}.$$
Similarly
$$v\cdot (u\cdot a)=\{v, [u,\d(a)]\}=-\{v,[\d(a),u] \}.$$
Now we can use (\ref{23}) to write
$$u\cdot (v\cdot a)-v\cdot (u\cdot a)-[u,v]\cdot a=\{u, [v,\d(a)]\}+\{v,[\d(a),u]+\{\d(a), [u,v]\} \}=0$$
and the identity (\ref{221}) follows.

For the identity (\ref{222}), observe that we have
$$x\cdot [a,b]=\{x,\d [a,b]\} =\{x,[\d (a),\d (b)]\}.$$
We also have
$$[x\cdot a,b]=\{\d(x \cdot a),\d(b)\}=\{ \d\{x,\d(a)\},\d(b)\}=-\{ \d(b),[x,\d(a)]\}
.$$
Similarly,
$$[a,x\cdot b]=-\{\d(a), [\d(b),x]\} $$
and the identity (\ref{222}) also follows from (\ref{23}).

We still need to check the relations (\ref{18}) and (\ref{19}) in our case. We have
$$\d (x\cdot a)=\d\{x,\d(a)\}=[x,\d(a)]$$
and (\ref{18}) is proved. Finally we have
$$\d(a)\cdot b=\{\d(a),\d(b)\}=[a,b]$$
and (\ref{19}) is proved. 
	
\item [ii)]  By (\ref{20}), the Lie algebra $\pi_0(L_*)$ is abelian and by part i) the action of $L_0$ on $\pi_1(L_*)=Ker(\d)$ is trivial.
\end{enumerate}
\end{proof}
\subsection{Alternative definitions}\label{alternative_def}

\begin{Le}\label{20le} A crossed module of Lie algebras can equivalently be defined as a linear map $\partial:L_1\to L_0$, where $L_0$ is a Lie algebra, $L_1$ is a $L_0$-module and $\partial$ is a homomorphism of $L_0$-modules (that is, the equality (\ref{18}) holds) satisfying additionally the following identity 
\begin{equation} \label{20l}\partial(a)\cdot b+\partial(b)\cdot a=0,\end{equation}
for all $a,b\in L_1$.
\end{Le}

\begin{proof} By forgetting the Lie algebra structure on $L_1$, we see that any crossed module gives rise to a structure as described in the Lemma. Conversely, we can uniquely reconstruct the bracket on $L_1$ from such a structure by $[a,b]:=\partial(a)\cdot b$. Our first claim is that $\partial $ respects the bracket:
$$\partial ([a,b])=\partial (\partial (a)\cdot b)=[\partial(a),\partial (b)].$$
Here we used the identity (\ref{18}) for $x=\partial(a)$.
Our second claim is that $L_1$ is a Lie algebra. In fact, we have	
	$$[a,b]+[b,a]=\partial(a)\cdot b+\partial(b)\cdot a=0,$$
showing that the bracket is anticommutative. For the Jacobi identity we have
\begin{align*}
	[a,[b,c]]+[c,[a,b]]+[b,[c,a]]&=[a,[b,c]]-[[a,b],c]-[b,[a,c]]\\
	&=\partial (a)\cdot (\partial(b)\cdot c )-\partial ([a,b])\cdot c- \partial (b)\cdot (\partial(a)\cdot c ).
\end{align*}
Since 
\begin{align*}
\partial ([a,b])\cdot c&=\partial (\partial a \cdot b)\cdot c\\ &=[\partial (a),\partial (b)]\cdot c\\ &=\partial (a)\cdot (\partial(b)\cdot c )-\partial (b)\cdot (\partial(a)\cdot c ),
\end{align*}
we see that $[a,[b,c]]+[c,[a,b]]+[b,[c,a]]=0$. Thus, $L_1$ is also a Lie algebra and the second claim is proved.

It remains to show that $L_0$ acts on $L_1$ as a Lie algebra. Indeed, we have
$$x\cdot [a,b]-[x\cdot a,b]-[a,x\cdot b]
=x\cdot (\partial(a)\cdot b)-(\partial(x\cdot a))\cdot b-\partial (a)\cdot (x\cdot b).$$
Now we can observe that 
$$(\partial(x\cdot a))\cdot b=[x,\partial (a)]\cdot b=x\cdot( \partial(a)\cdot b)-\partial(a)(x\cdot b).$$
Comparing these expressions we see that $x\cdot [a,b]=[x\cdot a,b]+[a,x\cdot b]$ and the result follows.
\end{proof}

\begin{Pro}\label{22p} A braided crossed module of Lie algebras can equivalently be defined as a linear map of vector spaces $\partial:L_1\to L_0$ together with a bilinear map
	$$L_0\times L_0\to L_1,\ \ (x,y)\mapsto \{x,y\}$$ satisfying the identity (\ref{22}) and also the following identities
	\begin{equation} \label{24}
	\partial \{x,x\}=0,
	\end{equation}
\begin{equation} \label{26}
\{u,\partial \{v,w\}\}+\{w,\partial \{u,v\}\}+\{v,\partial \{w,u\}\}=0,
\end{equation}
where $a\in L_1$ and $x,u,v,w\in L_0$.
\end{Pro}
\begin{proof} 
By forgetting the Lie algebra structures on $L_0$ and $L_1$, we obtain a structure as described in the Proposition. In fact, we only need to check that the identities (\ref{24}) and (\ref{26}) hold. The identity (\ref{24}) follows from the fact that $[x,x]=0$ as in any Lie algebra and the identity (\ref{20}). Quite similarly, (\ref{26}) follows from (\ref{24}) and the identity (\ref{20}).

Conversely, assume $\partial:L_1\to L_0$ is equipped with a structure as described in the Proposition. Then we can define brackets on $L_0$ and $L_1$ based on the identities (\ref{20}) and (\ref{21}). In this way one obtains Lie algebras $L_0$ and $L_1$. In fact the Jacobi identity in both cases is a consequence of (\ref{26}). The antisymmetry of the bracket for $L_0$ follows from (\ref{24}). The antisymmetry for $L_1$ follows from (\ref{22}) by taking $x=\partial a$. 
Next, we  have
$\partial[a,b]=\partial\{\partial(b),\partial(a)\}$ by the definition of the bracket on $L_1$. The last expression is the same as $[\partial(a),\partial(b)]$ by the definition of the bracket on $L_0$. Thus, $\partial[a,b]=[\partial(a),\partial(b)]$. Hence, $\partial$ is a homomorphism of Lie algebras. By our constructions and assumptions, the identities (\ref{20}), (\ref{21}) and (\ref{22}) hold. Finally, the identity (\ref{23}) follows from (\ref{26}).
\end{proof}
\section{The centre of  a crossed module of Lie algebras} 
\subsection{Definition and the first properties}

The following is an analogue of the corresponding notion from \cite{mp_cent}.

\begin{De}\label{zl}Let $\partial: L_1\to L_0$ be a crossed module of Lie algebras. Denote by ${\bf Z}_0(L_*)$ the set of all pairs $(x, \xi)$, where $x\in L_0$ and $\xi:L_0\to L_1$ is a linear map satisfying the following identities 
	\begin{align}
	\d \xi(t)&=[x,t], \ \  \label{ZE1L}\\
	\xi (\d a )&= x \cdot a,  \  \label{ZE2L}\\
	\xi([s,t])&=  s\cdot \xi(t)-t\cdot \xi(s) .  \  \label{ZE3L}
	\end{align}
	Here $x,s,t\in L_0$ and $a\in L_1$.
\end{De}

\begin{Pro} \label{clb} \begin{enumerate}[wide, labelwidth=!, labelindent=0pt]
		\item [i)]  Let $c\in L_1$. Then the pair $(\partial(c),\xi_c)\in {\bf Z}_0(L_*)$, where
	$$\xi_c(t)=-t\cdot c.$$
\item [ii)] The linear map
$$\dd: L_1\to {\bf Z}_0(L_*)$$
given by $\dd(c)=(\partial(c),\xi_c)$, together with the bilinear map
$${\bf Z}_0(L_*) \times {\bf Z}_0(L_*) \to L_1$$
given by $$\{ (x,\xi),(y,\eta) \}= \xi(y),$$ is a braided crossed module of Lie algebras.
	\end{enumerate}
\end{Pro}

\begin{proof} 

\begin{enumerate}[wide, labelwidth=!, labelindent=0pt]
	\item [i)]
According to (\ref{18}) we have
$$\d \xi_c(t)=\d (-t\cdot c)=-[t,\d c]= [\d c,t]$$
and the equality (\ref{ZE1L}) holds. Next, we have
$$\xi_c(\d(a))=- \d (a)\cdot c=\d(c)\cdot a$$
and (\ref{ZE2L})  follows. Here we used the equality  (\ref{20l}). Finally, we have
$$\xi_c([s,t])=-[s,t]\cdot c= -(s\cdot (t \cdot c))+(t\cdot (s\cdot c))=s\cdot \xi_c(t)-t\cdot \xi_c(s).$$

	\item [ii)] We will use the definition of BCM given in Proposition \ref{22p}. Thus we only need to check the three identities listed in Proposition \ref{22p}. By definition we have
	$$\partial(\{(x,\xi),(x,\xi)\})=\partial(\xi(x))=[x,x]=0.$$
Here we used the identity (\ref{ZE1L}). Hence the identity (\ref{24}) holds. We also have
$$\{(x,\xi), \dd(a)\}+\{\dd(a),(x,\xi)\}=\xi(\d(a))+\xi_a(x)=
x\cdot a -x\cdot a=0.$$
This shows that the identity (\ref{22}) holds. Now we will prove the validity of the identity (\ref{26}). To this end, we take three elements in ${\bf Z}_0(L_*)$:
$$u=(x,\xi), \quad v=(y,\eta), \quad  w=(z,\zeta).$$
We have
$$\{v,w\}=\{(y,\eta),(z,\zeta)\}=\eta(z).$$
Hence $$\dd(\{v,w\})=(\d(\eta(z)), t\mapsto -t\cdot \eta(z))=([y,z],t\mapsto -t\cdot \eta(z)).$$
Thus,
$$\{u,\dd\{v,w\}\}=\xi([y,z]).$$
Similarly, $$ \{w,\dd \{u,v\}\}=\zeta([x,y]) \quad {\rm and} \quad \{v,\dd \{w,u\}\}=\eta[z,x].$$
Now we use the identity (\ref {ZE3L}) to obtain
$$\{u,\dd\{v,w\}\}=y\cdot \xi(z)-z\cdot \xi(y).$$
Using the identities (\ref {ZE2L}) and (\ref {ZE1L}), we can rewrite this as
$$\{u,\dd\{v,w\}\}=\eta(\partial(\xi(z)))-\zeta(\d \xi(y))=\eta([x,z])-\zeta([x,y]).$$
Thus we proved that
$$\{u,\dd\{v,w\}\}=-\{v,\dd \{w,u\}\}-\{w,\dd \{u,v\}\}$$
and (\ref{26}) follows.
\end{enumerate}
\end{proof}
\begin{De}
Let $L_*$ be a crossed module of Lie algebras. 
The BCM of Lie algebras $${\bf Z}_*(L_*)=(L_1\xto{\dd}{\bf Z}_0(L_*) )$$ constructed in Theorem \ref{clb} is called the \emph{centre} of $L_*$.
\end{De}
\subsection{The Lie algebra structure on ${\bf Z}_0(L_*) $}
\begin{Co} Let $L_*$ be a crossed module of Lie algebras. Then $$[(x,\xi),(y,\eta)]=([x,y],\theta)$$defines a Lie algebra structure on ${\bf Z}_0(L_*) $, where
$$\theta(t)= -t\cdot \xi(y)$$
and the map $\d:L_1\to L_0$ is a homomorphism of Lie algebras, where as usual
$$[a,b]=\d(a)\cdot b.$$
\end{Co}
\begin{proof} According to Proposition \ref{22p}, we obtain Lie algebra structures on $L_1$ and ${\bf Z}_0(L_*) $ by putting 
	$$[a,b]_{new}:=\{\dd(a),\dd(b)\}$$
	and
$$[(x,\xi),(y,\eta)]: =\dd \{(x,\xi),(y,\eta)\}.$$
Here we used  $[a,b]_{new}$ to denote the Lie algebra structure obtained from the bracket $\{ -,- \}$. Now we prove that the two Lie algebra structures on $L_1$ coincide. In fact, we have
$$[a,b]_{new}=\{(\d(a),\xi_a),(\d(b),\xi_b)\}=\xi_a(\d(b))=-\d(b)\cdot a=-[b,a]=[a,b].$$
Now we identify the bracket on ${\bf Z}_0(L_*) $. We have
$$[(x,\xi),(y,\eta)] =\dd (\xi(y))=(\d \xi(y), \xi_{\xi(y)}).$$
Since
$\d \xi(y)=[x,y]$ and
$$\xi_{\xi(y)}(t)=-t\cdot \xi(y),$$
the result follows.
\end{proof}

The following lemma shows that the same Lie algebra structure can be written in a slightly different form.

\begin{Le}\label{25le} If $(x,\xi)$ and $(y,\eta)$ are elements of ${\bf Z}_0(L_*) $, then
$$-t\cdot \xi(y)=\xi([y,t])-\eta([x,t])=t\cdot \eta(x).$$
\end{Le}
\begin{proof} We have
\begin{align*}
\xi([y,t])-\eta([x,t]) &= y\cdot \xi(t)-t\cdot \xi(y)-\eta([x,t]) \\ &= \eta(\d(\xi(t)))-t\cdot \xi(y)-\eta([x,t])\\ &= -t\cdot \xi(y).
\end{align*}
Here we first used the equality (\ref{ZE3L}), then (\ref{ZE2L}) and finally (\ref{ZE1L}). Quite similarly, we have
\begin{align*}
\xi([y,t])-\eta([x,t])&=\xi([y,t])-x\cdot \eta(t)+t\cdot \eta(x)\\ &= \xi([y,t])-\xi(\d(\eta(t)))+t\cdot \eta(x)\\ &=t\cdot \eta(x)
\end{align*}
and the proof is finished.
\end{proof} 
\subsection{The crossed module $L_*//{\bf Z}_*(L_*)$}

\begin{Pro} \label{11} \begin{enumerate}[wide, labelwidth=!, labelindent=0pt]
		 \item [i)] Let $y\in L_0$ and $(x,\xi)\in {\bf Z}_0(L_*) $. Then
$([y,x],\psi)\in {\bf Z}_0(L_*) $, where 
$$\psi(t)=t\cdot \xi(y).$$

\item [ii)] The rule  $$y\cdot (x,\xi):= ([y,x],\psi)$$
defines an action of $L_0$ on ${\bf Z}_0(L_*) $. With this action, the map
$${\bf Z}_0(L_*) \to L_0; \ \  (x,\xi)\mapsto x$$
is a crossed module of Lie algebras. 
\end{enumerate} 
\end{Pro}

\begin{proof} 
	
\begin{enumerate}[wide, labelwidth=!, labelindent=0pt]
	\item [i)]
We need to check that the pair $([y,x],\psi)$ satisfies the relations (\ref{ZE1L})-(\ref{ZE3L}). We have
	$$\d \psi(t)=\d (t\cdot \xi(y))=[t,\d (\xi(y))]=[t, [x,y]]=[[y,x],t]$$
and the relation (\ref{ZE1L}) holds. Next, we have
\begin{align*}
\psi(\d (a))&=\d(a) \cdot \xi(y)\\
&=\xi([\d(a),y])+y\cdot \xi(\d(a))\\
&=-\xi(\partial(y\cdot a))+y\cdot (x\cdot a)\\
& =-x\cdot (y\cdot a)+y\cdot (x\cdot a)\\
&=[y,x]\cdot a
\end{align*}
and the relation  	(\ref{ZE2L}) holds. Finally, we have
\begin{align*}\psi([s,t])&=[s,t]\cdot \xi(y)\\&=s\cdot (t\cdot \xi(y))-t\cdot (s\cdot\xi(y) )\\
&= s\cdot \psi(t)-t\cdot \psi(s)
\end{align*}
and 	(\ref{ZE3L}) holds.

\item [ii)] Let us first check that the above formula defines an action. This requires checking of two identities.
%
We have
\begin{align*}
u\cdot (v\cdot (x,\xi)) &= u\cdot ([v,x],t \mapsto t\cdot \xi(v) )\\
&=([u,[v,x]], t\mapsto t \cdot (u\cdot \xi(v)) ).
\end{align*}
Similarly, $v\cdot (u\cdot (x,\xi))=([v,[u,x]], t\mapsto t \cdot (v\cdot \xi(u)) )$.
Hence, we obtain
\begin{align*}u\cdot (v\cdot (x,\xi))-v\cdot (u\cdot (x,\xi))
&=([u,[v,x]], t\mapsto t \cdot (u\cdot \xi(v)) )- ([v,[u,x]], t\mapsto t \cdot (v\cdot \xi(u)) )\\&=([u,[v,x]]-[v,[u,x]], t\mapsto (t \cdot (u\cdot \xi(v))- t \cdot (v\cdot \xi(u)))\\&=([[u,v],x], t\mapsto t\cdot \xi([u,v]))
\end{align*}
and thus 
$$[u,v] \cdot (x,\xi)=u\cdot (v\cdot (x,\xi)) -v\cdot (u\cdot (x,\xi)).$$
For the second identity, observe that
\begin{align*}
y\cdot [(x,\xi),(x',\xi')] &= y\cdot ([x,x'],t\mapsto -t\cdot \xi(x'))\\
&=([y,[x,x']], t\mapsto -t\cdot (y\cdot \xi(x'))).
\end{align*}
On the other hand, $[y\cdot(x,\xi),(x',\xi')]+[(x,\xi),y\cdot (x',\xi')]$ is equal to 
$$
 ([[y,x], x'], t\mapsto -t\cdot (x'\cdot \xi(y)))+([x,[y,x']],t\mapsto -t\cdot \xi([y,x'])).$$
 The first coordinate of this expression is $$[[y,x], x']-[[y,x'],x]= [y,[x,x']],$$
 while the second coordinate is equal to
 \begin{align*}
t&\mapsto -t\cdot (x'\cdot \xi(y))-t\cdot \xi([y,x']))\\ &= -t\cdot (x'\cdot \xi(y))-t\cdot(y\cdot \xi(x'))+t\cdot (x'\cdot \xi(y))\\ &=-t\cdot(y\cdot \xi(x')).
 \end{align*}
Comparing these expressions, we see that
$$y\cdot [(x,\xi),(x',\xi')]=[y\cdot(x,\xi),(x',\xi')]+[(x,\xi),y\cdot (x',\xi')].$$
Hence $L_0$ acts on ${\bf Z}_0(L_*)$.

Thanks to Lemma \ref{20le}, it suffices to check the identities  (\ref{18}) and (\ref{20l}). Since the image of $y\cdot (x,\xi)=([y,x],\psi)$ into $L_0$ is $[y,x]$, the identity (\ref{18}) holds for  ${\bf Z}_0(L_*) \to L_0$. Finally for $A=(x,\xi)$, $A'=(x',\xi')$ we have
$$x\cdot A'+x'\cdot A=([x,x'],t\mapsto t\cdot \xi'(x))+
([x',x],t\mapsto t\cdot \xi(x'))=0$$
thanks to Lemma \ref{25le}. So the identity (\ref{20l}) also holds for ${\bf Z}_0(L_*) \to L_0$  and hence the result.

\end{enumerate} 
\end{proof}

\subsection{On homotopy groups of ${\bf Z}_*(L_*)$}\label{042se} Let $L_*$ be a crossed module. In this section we investigate $\pi_i$  of the crossed module ${\bf Z}_*(L_*)$, $i=0,1$. 
The case $i=1$ is easy and the answer is given by the following lemma.
\begin{Le} Let $L_*$ be a crossed module. Then
	$$\pi_1({\bf Z}_*(L_*) )\cong \H^0(\pi_0(L_*),\pi_1(L_*)).$$
\end{Le}
\begin{proof} By definition $a\in \pi_1({\bf Z}_*(L_*) )$ iff $\dd(a)=(0,0)$, thus when $\d (a)=0$ and $\xi_a(t)=0$ for all $t\in \G_0$. These conditions are equivalent to the conditions $a\in \pi_1(L_*)$ and $t\cdot a=0$ for all $t\in L_0$ and hence the result.
\end{proof}

To state our result on $\pi_0({\bf Z}_*(L_*))$ we need to fix some notation. 

For a Lie algebra  $L$, we let ${\sf Z}(L)$ denote the centre of $L$, which is the set of elements $c\in L$ for which $[c,x]=0$ for all $x\in L$. Moreover, if $K$ is a Lie algebra on which $L$ acts, we set
$$st_K(L)=\{x\in L|\, x\cdot k=0 \ {\rm for \ all}\ k\in K\}.$$

The intersection of ${\sf Z}(L)$ and $st_K(L)$ is denoted by ${\sf Z}_{K}(L)$. Thus $c\in {\sf Z}_{K}(L)$ iff $c\cdot k=0$ and $[x,c]=0$ for all $x\in L$ and $k\in K$.

Let $L_*$ be a crossed module. In this case we have defined
$${\sf Z}_{\pi_1(L_*)}(\pi_0(L_*)) \ \ {\rm and}
\ \  \ {\sf Z}_{L_1}(L_0).$$
We will come back to the second one in the next section. Now we relate ${\sf Z}_{\pi_1(L_*)}(\pi_0(L_*))$ to $\pi_0({\bf Z}_*(L_*))$. To this end denote the class of $x\in L_0$ in $\pi_0(L_*)$ by  ${\sf cl}(x)$. Take now an element $(x,\xi)\in {\bf Z}_0(L_*)$. Accordingly, ${\sf cl}(x,\xi)$ denotes the class of $(x,\xi)$ in $\pi_0({\bf Z}_*(L_*))$. 

\begin{Le} \label{14l} \begin{itemize}[wide, labelwidth=!, labelindent=0pt]
\item [i)] Let $(x,\xi)\in {\bf Z}_0(L_*)$. Then the class ${\sf cl}(x)\in \pi_0(L_*)$ belongs to $${\sf Z}_{\pi_1(L_*)}(\pi_0(L_*)).$$ 
Thus one has the well defined homomorphism
$${\bf Z}_0(L_*)\xto{\omega'} {\sf Z}_{\pi_1(L_*)}(\pi_0(L_*))$$
given by $\omega'(x,\xi):={\sf cl}(x).$

\item [ii)] The composite map
	$$L_1\xto{\dd} {\bf Z}_0(L_*)\xto{\omega'} {\sf Z}_{\pi_1(L_*)}(\pi_0(L_*))$$
	is trivial. Hence the map $\omega'$ induces a  homomorphism
	$$\omega: \pi_0({\bf Z}_*(L_*))\to  {\sf Z}_{\pi_1(L_*)}(\pi_0(L_*)).$$
	
\item [iii)] For a $1$-cocycle $\phi: \pi_0(L_*)\to \pi_1(L_*)$, the pair $(0,\tilde{\phi})\in {\bf Z}_0(L_*)$, where $\tilde{\phi}:L_0\to L_1$ is the composite map
$$L_0\to \pi_0(L_*)\xto{\phi} \pi_1(L_*)\hookrightarrow L_1.$$
Moreover, the assignment $\phi\mapsto {\sf cl}(0,\tilde{\phi})$ induces a  homomorphism
$$f:   H^1(\pi_0(L_*), \pi_1(L_*))\to \pi_0({\bf Z}_*(L_*)).$$

\item [iv)]  These maps fit in an exact sequence
$$0\to \H^1(\pi_0(L_*), \pi_1(L_*))\xto{f} \pi_0({\bf Z}_*(L_*))\xto{\omega}  {\sf Z}_{\pi_1(L_*)}(\pi_0(L_*)).$$
\end{itemize}
\end{Le}

\begin{proof} \begin{itemize}[wide, labelwidth=!, labelindent=0pt]

\item [i)] In fact, according to the equation (\ref{ZE1L}) of Definition \ref{zl}, $[x,y]=\d \xi(y)$ and hence ${\sf cl}([x,y])=0$ in  $\pi_0(L_*)$. It follows that ${\sf cl}(x)$ is central in $\pi_0(L_*)$. Moreover, the equation (\ref{ZE2L}) of Definition \ref{zl} tells us that $\xi(\d a)=x\cdot a$. In particular, if $a\in\pi_1(L_*)$ (i.e. $\d (a)=0$) then $x\cdot a=0$ and the result follows. 

\item [ii)]   Take $a\in L_1$. By construction $\dd(a)=(\d(a),\xi_a)$. Hence $$\omega'\dd (a) ={\sf cl}(\d(a))=0.$$ 

\item [iii)] Since $\phi$ is a 1-cocycle, $\tilde{\phi}$ satisfies the condition (\ref{ZE3L}) of Definition \ref{zl}. Next, the values of $\tilde{\phi}$ belong to $\pi_1(L_*)$, so $\d \tilde{\phi}=0$ and the condition (\ref{ZE1L}) follows. Finally, $\tilde{\phi}(\d a)=\phi({\sf cl}(\d a))=\phi(0)=0$ and (\ref{ZE2L}) also holds. It remains to show that if $\phi(t)=t\cdot b$, 
for an element $b\in\pi_1(L_*)$, then ${\sf cl}(0,\tilde{\phi})=0$, but this follows from the fact that $\dd(b)=(0,\tilde{\phi}).$

\item [iv)] Exactness at $\H^1(\pi_0(L_*), \pi_1(L_*))$: Assume $\phi:\pi_0(L_*)\to \pi_1(L_*)$ is a 1-cocycle such that ${\sf cl}(0,\tilde{\phi})$ is the trivial element in $\pi_0({\bf Z}_*(L_*))$. That is, there exists a $c\in L_1$ such that $\dd(c)=(0,\tilde{\phi})$. Thus $\d (c)=0$ and $\tilde{\phi}(t)=-t\cdot c$.
So $c\in \pi_1(L_*)$ and the second equality implies that the class of \(\phi\) is zero in $\H^1(\pi_0(L_*), \pi_1(L_*))$, proving that $f$ is a monomorphism.

Exactness at  $\pi_0({\bf Z}_*(L_*))$: First take a cocycle $\phi:\pi_0(L_*)\to \pi_1(L_*)$. Then
$$\omega\circ f({\sf cl}(\phi)) =\omega({\sf cl}(0,\tilde{\phi}))
={\sf cl}(0)=0.$$
Take now an element $(x,\xi)\in {\bf Z}_0(L_*)$ such that ${\sf cl}(x,\xi)\in \ker(\omega)$. Thus $x=\d(a)$ for $a\in L_1$. Then ${\sf cl}(x,\xi)={\sf cl}(y,\eta)$, where \[(y,\eta)=(x,\xi)-\dd(a) = (x,\xi)-(\d(a), t\mapsto -t\cdot a)= (x-\d(a), \xi(t)+t\cdot a).\] Since $y=x-\d (a)=0$, we see that $\d \eta(t)=[y,t]=[0,t]=0$ and \[\eta(\d c)=\xi(\d(c))+\d(c)\cdot a = x\cdot c+\d(c)\cdot a=0\] as \(x=\d(a)\). So $\eta= \tilde{\phi}$, where $\phi:\pi_0(L_*)\to \pi_1(L_*)$ is a 1-cocycle. Thus $f({\sf cl}(\phi))={\sf cl}(0,\eta)={\sf cl}(x,\xi)$ and exactness at $\pi_0({\bf Z}_*(L_*))$ follows. 

\end{itemize}	
\end{proof}

 The map $\omega$ is not surjective in general. That is, for an element $x\in L_0$ such that ${\sf cl}(x)\in {\sf Z}_{\pi_1(L_*)}(\pi_0(L_*))$, there is in general no linear map $\psi:L_0\to L_1$ satisfying the conditions (\ref{ZE1L})-(\ref{ZE3L}) of the Definition \ref{zl}. However, there is a function satisfying (\ref{ZE1L}) and (\ref{ZE2L}), see the following Proposition. We will use this observation to extend the exact sequence constructed in Lemma \ref{14l}.

\begin{Pro}\label{15pr} 
\begin{enumerate}[wide, labelwidth=!, labelindent=0pt]
	
\item [i)] Take an element $x\in L_0$ such that ${\sf cl}(x)\in {\sf Z}_{\pi_1(L_*)}(\pi_0(L_*))$. Then there exists a linear map $\psi:L_0\to L_1$ such that $\d \psi(t)=[x,t]$ and $\psi(\d a)=x\cdot a $ for all $t\in L_0$ and $a\in L_1$. 

\item [ii)] The expression
\[\bar{\theta}(s,t):=s\cdot \psi(t)- t\cdot \psi(s)-\psi([s,t])\]
is skew-symmetric, belongs to $\pi_1(L_*)$ and vanishes if \(s = \d(a)\) or \(t=\d(a)\) for some \(a\in L_1\). Hence, 
\(\bar{\theta}\) defines a map \(\Lambda ^2\pi_0(L_*)\xrightarrow{\tilde{\theta}} \pi_1(L_*)\).

\item [iii)] The map \(\tilde{\theta}\) is a $2$-cocycle. That is, it satisfies the relation
\[s\cdot \tilde{\theta}(t,r) -t\cdot \tilde{\theta}(s,r)+r\cdot \tilde{\theta}(s,t)=\tilde{\theta}([s,t],r)-\tilde{\theta}([s,r],t)+\tilde{\theta}([t,r],s).\]

\item [iv)] The class $\theta\in \H^2(\pi_0(L_*),\pi_1(L_*))$ is independent of the choice of $\psi$. The assignment $x\mapsto \theta$ defines a group homomorphism
$$g: {\sf Z}_{\pi_1(L_*)}(\pi_0(L_*))\to \H^2(\pi_0(L_*), \pi_1(L_*)) .$$
\end{enumerate}
\end{Pro}

\begin{proof}  

\begin{enumerate}[wide, labelwidth=!, labelindent=0pt]
	
\item [i)] Since \({\sf cl}(x)\) is central in \(\pi_0(L_*)\), for each $t\in L_0$ there exists
an element $a_t\in L_1$ such that $\d(a_t)=[x,t]$. We choose a linear splitting $L_0=\im(\d)\oplus T$ and a linear basis $(t_i)$ of $T$. The assignment $t_i\mapsto a_{t_i}$ can be extended as a linear map $\psi:T\to L_1$ for which $\d(\psi(t))=[x,t]$ holds for all $t\in T$. To define $\psi$ on whole $L_0$, we first take $x\in\im(\d)$. In this case $x=\d a$ and we can set $\psi(x)= x\cdot  a$. This is well-defined since if $\d a= \d b$ we will have $a=b+c$, where $\d c=0$. It follows that $c\in \pi_1(L_*)$. By our assumption on $x$ we also have $x\cdot c=0$. It follows that $x\cdot a=x\cdot b$. This shows that $\psi$ is well defined on $\im(\d)$.
Since $L_0$ is the direct sum of $T$ and $\im(d)$ we can extend $\psi$ uniquely on $L_0$. It remains to check that $\d(\psi(t))=[x,t]$ holds for all $t\in L_0$.  The required identity is linear on $t$ and since it holds when $t\in T$, we can assume that $t=\d(a)$. In this case we have
$$\d(\psi(t))=\d(\psi(\d(a))=\d(x\cdot a)=[x,\d(a)]=[x,t]$$
and the result follows.

\item [ii)] 
We have
\begin{align*} \bar{\theta}(t,s)&=t\cdot \psi(s)-s\cdot\psi(t)-\psi([t,s])\\
	&= -(s\cdot \psi(t)-t\cdot\psi(s)-\psi([s,t])= -\bar{\theta}(s,t).
\end{align*}
We also have
\begin{align*} \d(\bar{\theta}(s,t))&=\d (s\cdot \psi(t)-t\cdot\psi(s)-\psi([s,t]))\\
&= [s,[x,t]]-[t,[x,s]]-[x,[s,t]]\\
&=0.
\end{align*}
Finally, let \(t=\d(a)\) for some \(a\in L_1\). Then
\[\bar{\theta}(s, \d(a))=s\cdot \psi(\d(a))-\d(a)\cdot\psi(s)-\psi([s, \d(a)]).\]
We have \(\psi(\d(a)) = x\cdot a\) because \(\psi\) satisfies equation (\ref{ZE2L}). Then by equation (\ref{20l}) and because \(\psi\) satisfies equation (\ref{ZE1L}), we have \(-\d(a)\cdot\psi(s) = \d(\psi(s))\cdot a = [x,s]\cdot a\). Finally, \(-\psi([s, \d(a)]) = -\psi(\d(s\cdot a)) = -x\cdot (s\cdot a)\) because \(\d\) is a Lie module homomorphism and \(\psi\) satisfies equation (\ref{ZE2L}). Thus,
\[\bar{\theta}(s, \d(a))=s\cdot (x\cdot a) + [x,s]\cdot a-x\cdot (s\cdot a)=0.\]
Hence, \(\tilde{\theta}:\Lambda ^2\pi_0(L_*)\to \pi_1(L_*)\) is a well-defined map. 

\item [iii)] The 2-cocycle condition follows from direct computation.

\item [iv)] If $\psi$ and $\psi'$ both satisfy the conditions in i), then $\psi-\psi'$ vanishes on $\im(\d)$ and lies in $\ker(\d)=\pi_1(L_*)$. Thus we obtain a well-defined map $h:\pi_0(L_*)\to \pi_1(L_*)$ such that  $\psi-\psi'$ is equal to the composite map $L_0\twoheadrightarrow \pi_0(L_*)\xto{h}\pi_1(L_*)\hookrightarrow L_1$. From this the result follows, because  $\bar{\theta}-\bar{\theta'}=d(h)$, where $\theta'$ is the function corresponding to $\psi'$ as in ii) and $$d:\hom(\pi_0(L_*),\pi_1(L_*))\to \hom(\Lambda ^2\pi_0(L_*),\pi_1(L_*))$$ is the coboundary map in the standard complex computing the Lie algebra cohomology. 
\end{enumerate}
\end{proof}

\begin{Co}\label{12} Let $L_*$ be a crossed module of Lie algebras. Then one has an exact sequence
$$0\to \H^1(\pi_0(L_0), \pi_1(L_*))\xto{f} \pi_0({\bf Z}_*(L_*))\xto{\omega}  {\sf Z}_{\pi_1(L_*)}(\pi_0(L_*)) \xto{g} \H^2(\pi_0(L_*), \pi_1(L_*)).$$
 
\end{Co} 
%
\begin{proof} According to Lemma \ref{14l} we only need to show exactness at ${\sf Z}_{\pi_1(L_*)}(\pi_0(L_*))$. To this end, take $(x,\xi)\in {\bf Z}_0(L_*)$. Since $\omega'(x,\xi)={\sf cl}(x)$, we can choose $\psi=\xi$ for $g({\sf cl}(x))$. Clearly $\bar{\theta}=0$ for this $\psi$ and hence $g\circ \omega=0$. 

Take now $x\in L_0$ such that ${\sf cl}(x)\in {\sf Z}_{\pi_1(L_*)}(\pi_0(L_*))$. Assume $g({\sf cl}(x))=0$. We have to show that there exists a map $\xi:L_0\to L_1$ such that $(x,\xi) \in 
{\bf Z}_0(L_*)$. According to part i) of Proposition \ref{15pr}, we can choose a map $\psi:L_0\to L_1$ such that the pair $(x,\psi)$ satisfies all requirements in Definition \ref{zl} except perhaps the relation (\ref{ZE3L}). By construction, $g(x)$ is the class of the 2-cocycle $\bar{\theta}$ in $\H^2(\pi_0(L_*), \pi_1(L_*))$. As this class is zero, $\bar{\theta}$ is a coboundary and therefore there exists a function $\phi: \pi_0(L_*)\to \pi_1(L_*)$ for which $\bar{\theta} (s,t)=s\cdot \phi(t)-t \cdot \phi(s)-\phi([s,t])$.
Since $\phi$ takes values in $\pi_1(L_*)$, the map $\psi'(t)=\psi(t)-\phi(t)$ also satisfies the conditions in iii), thus we can replace $\psi$ by $\psi'$. The above equation shows that the corresponding $2$-cocycle $\bar{\theta'}$ vanishes, meaning that $\psi'$ is a $1$-cocycle. Thus $(x,\psi')\in {\bf Z}_0(L_*)$ and exactness follows.  

\end{proof}

\subsection{Relation to nonabelian cohomology}\label{43se} In this section we relate our centre to the  non-abelian cohomology of Lie algebras developed by D. Guin \cite{dg1}.

Let $L_*$ be a crossed module. The group $\H^0(L_0,L_*)$ is defined by
$$\H^0(L_0,L_*)=\{a\in L_1|\,\d a=0 \ {\rm and} \  x\cdot a=0\ {\rm for \ all}\ x\in L_0\}.$$
Clearly, $\H^0(L_0,L_*)=\H^0(\pi_0(L_*),\pi_1(L_*))$.
It is a central subgroup of $L_1$.

In order to define the first cohomology group $\H^1(L_0,L_*)$, we first introduce the group $\Der_{L_0}(L_0,L_1)$. Elements of $\Der_{L_0}(L_0,L_1)$ are pairs $(g,\gamma)$ (see \cite[Definition 2.2.]{dg1}), where $g\in L_0$ and $\gamma:L_0\to L_1$ is a function for which two conditions hold:
$$\gamma([s,t])=s\cdot \gamma(t)-t\cdot \gamma (s)  \quad  {\rm and} \quad  \d \gamma(t)=[g, t].$$ Here $g,s,t\in L_0$. Comparing with Definition \ref{zl}, we see that these are exactly the conditions \ref{ZE1L} and \ref{ZE3L} of Definition \ref{zl}. Hence ${\bf Z}_0(L_*) \subset \Der_{L_0}(L_0,L_1)
$. 

Define a Lie algebra structure on $\Der_{L_0}(L_0,L_1)$ as follows.
If $(g,\gamma),(g',\gamma')\in \Der_{L_0}(L_0,L_1)$, then $([g,g'],\, \gamma\ast \gamma')\in \Der_{\G_0}(\G_0,\G_1)$ (see \cite[Lemme 2.3.1]{dg1}), 
where $\gamma\ast \gamma'$ is defined by
$$(\gamma\ast \gamma')(t)=\gamma(g'\cdot t)-\gamma'(g\cdot t).$$
Then ${\bf Z}_0(L_*)$ is a subalgebra of $\Der_{L_0}(L_0,L_1)$.

The group $\H^1(L_0,L_*)$ is defined as the quotient  $\Der_{L_0}(L_0,L_1)/I$, where 
$$I=\{(\d(a)+c,\eta_a)\}$$
Here $c$ is a central element of $L_0$, $a\in L_1$ and $\eta_a: L_0\to L_1$ is defined by
$$\eta_a(t)=t\cdot a.$$
The fact that $I$ is an ideal is proved in \cite[Lemma 2.4] {dg1}.
The following fact is a direct consequence of the definition.

\begin{Le} One has a commutative diagram with exact rows
$$\xymatrix{
	0\ar[r] &\H^0(L_0,L_*)\ar[r]& L_1\ar[rr]^\dd& &\Der_{L_0}(L_0,L_1)\ar[r] & \H^1(L_0,L_*)\ar[r] & 1 \\
	0\ar[r] & \pi_1({\bf Z}_*(L_*)) \ar[u]_{\cong }\ar[r]& L_1\ar[u]_{Id}\ar[rr]^\dd &&{\bf Z}_0(L_*)\ar[r] \ar@{^{(}->}[u]
	 & \pi_0({\bf Z}_*(L_*)) \ar@{^{(}->}[u] \ar[r]& 1
}
$$
\end{Le}

\begin{Rem}
Guin's definition \cite[Lemme 2.3.1]{dg1} of the Lie algebra structure on \(\Der_{L_0}(L_0,L_1)\) agrees with our definition of the Lie algebra structure on \({\bf Z}_0(L_*)\). Guin's definition directly translates to \[[(x,\xi),(y,\eta)] = \xi([y,t])-\eta([x,t]),\]
while from Proposition \ref{clb} we have
\[\{(x,\xi),(y,\eta)\} = \xi(y),\]
which lifts to
\[[(x,\xi),(y,\eta)] = -t\cdot\xi(y).\]
However, 

\[ \xi([y,t])-\eta([x,t]) = -t\cdot\xi(y),\]
according to Lemma \ref{25le} and so the definitions agree.

\end{Rem}

\section*{Addendum}
In \cite{mp_cent} we proved that the centre of a crossed module is intimately related to the Drinfeld centre  of a monoidal category \cite{QG}. Here we introduce the notion of a centre for Lie 2-algebras \cite{HAVI}, which are the Lie analogues of categorical groups. 
By definition a (strict) Lie $2$-algebra is a Lie algebra object in the category of all small categories. Thus a Lie $2$-algebra is a category $\bf L$ whose set of objects is denoted by ${\bf L}_0$ and whose set of arrows is denoted by ${\bf L}_1$. Two bifunctors are given:
the addition $+:{\bf L}\times {\bf L} \to {\bf L}$ and bracket $[-,]:{\bf L}\times {\bf L} \to {\bf L}$. Also, for any $k\in  K$ an endofunctor  $\lambda_
k:{\bf L} \to {\bf L}$ (the multiplication by a scalar) is given such that the all axioms of a Lie algebra hold strictly. For example, the distributivity law $k(x+y)=kx+ky$ implies the commutativity of the  diagram
$$\xymatrix{{\bf L}\times {\bf L}\ar[r]^{+}\ar[d]_{\lambda_k\times \lambda_k}& {\bf L}\ar[d]^{\lambda_k}\\
	{\bf L} \times {\bf L} \ar[r]^{+}& {\bf L}.
}$$

It is well-known that the category of Lie 2-algebras and the category of crossed modules are equivalent. We recall how to obtain Lie 2-algebras from crossed modules. 

Let $\d:X_1\to X_0$ be  a linear map. Then one has the category ${\sf Cat}(X_*)$. Objects are elements of $X_0$ and a morphism from $x$ to $y$ is given by the diagram $x\xto{a} y$, where $a\in X_1$ satisfies the condition $\d  (a)+x=y$. Sometimes this morphism is also denoted by $(x,a)$. Thus $(x,a):x\to \d (a)+x$. The composition law in  ${\sf Cat}(X_*)$ is induced by the addition in $L_1$. The identity morphism $\id_x$ of an object $x$ is $x\xto{0} x$. So $\id_x=(x,0)$. For any $k\in K$ we have an endofunctor $\lambda_k:{\sf Cat}(X_*) \to {\sf Cat}(X_*)$ which on objects is given by $\lambda_k(x)=kx$ and on morphisms it is given by
\[\lambda_k(x\xto{a} y)=kx\xto{ka} ky.\]
We also have the bifunctor \[+:{\sf Cat}(X_*) \times  {\sf Cat}(X_*) \to {\sf Cat}(X_*)\]
which on objects is given by $(x,y)\mapsto x+y$ and on morphisms by
\[(x\xto{a} y)+(x'\xto{a'} y)'=x+x'\xto{a+a'} y+y'.\]
In this way we obtain a strict $K$-vector space object in the category of small categories. 

In the case when instead of a linear map $\d:X_1\to X_0$, a crossed module of Lie algebras is given, there is a bifunctor $[-,-]:{\sf Cat}(L_*)\times {\sf Cat}(L_*)\to {\sf Cat}(L_*)$, which on objects is given by the Lie algebra structure on $L_0$ and on morphisms it assigns
to  $x\xto{a} y$ and $x'\xto{a'} y'$ the morphism
\[\xymatrix{[x,x']\ar[rrr]^{\d(a)\cdot a'-x'\cdot a+x\cdot a'}&&&
	[y,y'] }.\]

In particular, we have the morphism $[x,x']\xto{x\cdot a'} [x,y']$, which is also denoted by $[x,x']\xto{[\id_x,a']}[x,y']$.
In other words \[[(x,a),(x',a')]=([x,x'],\d(a)\cdot a'-x'\cdot a+x\cdot a').\]

For any $(x,\xi)\in {\bf Z}_0(L_*)$ and $y\in L_0$ one puts $\tau_y=-2\xi(y)$. Then $\tau_y$
defines a morphism $\tau_y:[x,y]\to [y,x]$. In fact, the collection $(\tau_y)_{y\in L_0}$
defines a natural transformation of endofunctors
$$\tau:[x,-]\to [-,x]$$
for which  additionally the following equality holds
$$\tau_{[y,z]}=y\cdot \tau_z-z\cdot \tau_y.$$
This equality can be rewritten as \begin{equation}\label{taufr}\tau_{[y,z]}=[\id_y,\tau_z]-[\id_z, \tau_y].\end{equation}

This suggests that we can define the centre of a Lie 2-algebra ${\bf L}$ to be the category whose objects are pairs $(x,\tau)$, where $x\in {\bf L}_0$ is an object of ${\bf L}$  and $\tau$ is a natural transformation $\tau:[x,-]\to [-,x]$, that is, the collection of morphisms $(\tau_y:[x,y]\to [y,x])_{y\in{\bf L}_0}$ such that for any morphism $a:y\to z$ one has a commutative diagram
$$\xymatrix{[x,y]\ar[r]^{\tau_y} \ar[d] _{[\id_x,a]}& [y,x] \ar[d]^{[a,\id_x]}\\
	[x,z]\ar[r]_{\tau_z}& [z,x]}.$$
One requires that additionally the equality (\ref{taufr}) holds. A morphism $(x,\tau)\to (x',\tau')$ is given by a morphism $a:x\to x'$ for which the following diagram commutes
$$\xymatrix{[x,y]\ar[r]^{\tau_y} \ar[d] _{[a,\id_y]}& [y,x] \ar[d]^{[\id_y,a]}\\
	[x',y]\ar[r]_{\tau'_y}& [y,x'].
}$$
Denote this category by ${\mathcal Z}({\bf L})$ and call it the centre of the Lie $2$-algebra ${\bf L}$. 
The bifunctor $$[-,-]:
{\mathcal Z}({\bf L})\times {\mathcal Z}({\bf L})\to {\mathcal Z}({\bf L})$$
is defined by
$$[(x,\tau),(y,\eta)]=([x,y],\theta)$$
where $$\theta_z:[[x,y],z]\to [z,[x,y]]$$
is given by
$$\theta_z=\tau_{[y,z]}-\eta_{[x,z]}.$$
Here we used the facts that
$$[[x,y],z]=[x,[y,z]]-[y,[x,z]] \quad {\rm and} \quad 
[z,[x,y]]= [[y,z],x]-[[x,z],y].$$
Our computations show that ${\mathcal Z}({\bf L})$  is again a Lie 2-algebra and in fact there is also braiding (see  \cite{ulu},\cite{manolo} for braided Lie 2-algebras) given by
\[ \tau_{[(x',\tau'), (x'',\tau'')]}=\tau'_y.\]

\end{document}